\documentclass[11pt,reqno]{amsart}

\usepackage{amsmath,amsthm,amsfonts,amssymb,latexsym,mathrsfs,color,pstricks}
\usepackage{hyperref}

\newtheorem{theorem}{Theorem}%[section]
\newtheorem{corollary}[theorem]{Corollary}

\newtheorem{problem}[theorem]{Open Problem}

\newtheorem{example}[theorem]{Example}

\newcommand{\Stirling}[2]{\genfrac{\{}{\}}{0pt}{}{#1}{#2}}

\newcommand{\Eulerian}[2]{\genfrac{<}{>}{0pt}{}{#1}{#2}}
\newcommand{\arxiv}[1]{\href{http://arxiv.org/abs/#1}{\texttt{arXiv:#1}}}
\newcommand{\luschny}{http://www.luschny.de/math/euler/StirlingFrobeniusNumbers.html}

\newcommand{\sgn}{{\rm sgn\,}}

\newcommand{\mn}{{\mathcal N}}
\title[On generalized Stirling grammars]{Normal ordering problem and the extensions of the Stirling grammar}

\author[S.-M.~Ma]{Shi-Mei~Ma}
\address{School of Mathematics and Statistics,
         Northeastern University at Qinhuangdao,
         Hebei 066004, P. R. China}
\email{shimeimapapers@gmail.com (S.-M. Ma)}
\author[T.~Mansour]{Toufik Mansour}
\address{Department of Mathematics, University of Haifa, 31905 Haifa, Israel}
\email{tmansour@univ.haifa.ac.il (T. Mansour)}
\author{Matthias Schork}
\address{Camillo-Sitte-Weg 25, 60488 Frankfurt, Germany}
\email{mschork@member.ams.org (M. Schork)}

\begin{document}

\maketitle

\begin{abstract}
The purpose of this paper is to investigate the connection between context-free grammars and normal ordering problem, and then to explore various extensions of the Stirling grammar.
We present grammatical characterizations of several well known combinatorial sequences, including
the generalized Stirling numbers of the second kind related to the normal ordering problem and
the $r$-Dowling polynomials. Also, possible avenues for future research are described.

\end{abstract}

\section{Introduction}\label{sec:Introduction}
Throughout this paper, we always assume that $n\geq 1$.
The {\em Weyl algebra} can be considered as an abstract algebra generated by two symbols $U$ and $V$
satisfying the commutation relation
\begin{align}\label{eq1}
UV-VU=1.
\end{align}
A well known representation of~\eqref{eq1} is by operators $V=X$ and $U=D$, where $X$ and $D$ are defined by $X(f)(x)=xf(x)$ and $D(f)(x)=\frac{df}{dx}(x)$, respectively. Clearly, $(DX-XD)f=f$ for any function $f$. By the commutation relation $DX=XD+1$, any word $\omega$ in $X$ and $D$ can be brought into {\em normal ordering form}
\begin{align}\label{eq2}
\mathcal{N}(\omega)=\sum_{i,j\geq0}c_{ij}(\omega)X^iD^j,
\end{align}
for some nonnegative integers $c_{ij}(\omega)$. For specific words $\omega$, the problem of normal ordering was already investigated by Scherk~\cite{S1823} in 1823. In particular, he showed for $\omega=(XD)^n$ that
\begin{equation}\label{Scherk-Stirling}
\mathcal{N}((XD)^n)=\sum_{k=0}^n\Stirling{n}{k}X^kD^k,
\end{equation}
where $\Stirling{n}{k}$ are the {\it Stirling numbers of the second kind}, which satisfy the recurrence relation
\begin{equation*}
\Stirling{n}{k}=k\Stirling{n-1}{k}+\Stirling{n-1}{k-1}
\end{equation*}
with the initial conditions $\Stirling{n}{1}=1$ and $\Stirling{n}{k}=0$ for $k>n$.
Since 1930 several other words have been investigated (for instance, $\omega=(X^rD^s)^n$ for $r,s\geq0$) and a combinatorial interpretation of the related normal ordering coefficients $c_{ij}(\omega)$ in~\eqref{eq2} has been given. For instance, in 2005, Varvak~\cite{Varvak05} established an interpretation as rook numbers of an associated {\em Ferrers board}. Other interpretations can be found in~\cite{AMS,Blasiak10,L2009,Mansour1101,Mansour1102,Mo2011,SB}. In particular, Lang~\cite{L2009} and Mohammad-Noori~\cite{Mo2011} considered generalized Stirling numbers and drew a connection to certain labeled trees. Copeland~\cite{cop} considered representations of powers of vector fields in terms of forests. In~\cite{AMS}, Asakly et al. showed how the normal ordered form of a word can be read off from a labeled tree.

In the physical context, the relation~\eqref{eq1} appears in quantum physics as the commutation relation for the single mode boson {\em annihilation operator} $b$ and {\em creation operator} $b^{\dag}$ satisfying $bb^{\dag}-b^{\dag}b=1$. The number operator $N=b^{\dag}b$ has a special importance in physics. In this context, normal ordering is a functional representation of operator functions in which all the creation operators stand to the left of the annihilation operators. The normal ordered form of functions in $b$ and $b^{\dag}$ allows to evaluate correlation functions in a simpler way and has been investigated since the beginning of quantum mechanics. Katriel~\cite{Katriel74} considered normal ordering $N^n=(b^{\dag}b)^n$ and showed that
\begin{equation}\label{boson-Stirling}
\mathcal{N}((b^{\dag}b)^n)=\sum_{k=0}^n\Stirling{n}{k}(b^{\dag})^kb^k.
\end{equation}
This relation (and also generalizations thereof) and the combinatorial nature of the corresponding coefficients has been studied in the
physical literature (see~\cite{Bla2004,Benoumhani97,Blasiak03,Blasiak10,Katriel74,Kat92,Mendez05,Mik85,Schork03,Schork06}).

Let $A$ be an alphabet whose letters are regarded as independent commutative indeterminates. A {\it context-free grammar} $G$ over $A$ is defined as a set of substitution rules that replace a letter from $A$ by a formal function over $A$. The {\it formal derivative} $D$ is a linear operator defined with respect to a context-free grammar $G$. For any formal functions $u$ and $v$, we have
$$D(u+v)=D(u)+D(v),\quad D(uv)=D(u)v+uD(v) \quad and\quad D(f(u))=\frac{\partial f(u)}{\partial u}D(u),$$
where $f$ is an analytic function. It follows from {\it Leibniz's formula} that
\begin{equation}\label{Dnab-Leib}
D^n(uv)=\sum_{k=0}^n\binom{n}{k}D^k(u)D^{n-k}(v).
\end{equation}
The grammatical method was systematically introduced by Chen~\cite{Chen93}
in the study of exponential structures in combinatorics. In particular,
Chen~\cite[Eq.~(4.8)]{Chen93} studied the {\it Stirling grammar}
\begin{equation}\label{Chen-grammar}
G=\{x\rightarrow xy, y\rightarrow y\},
\end{equation}
and he found that
\begin{equation}\label{Chen-grams}
D^n(x)=x\sum_{k=1}^n\Stirling{n}{k}y^k.
\end{equation}
Subsequently, Dumont~\cite{Dumont96} considered chains of general substitution rules on words. In particular,
Dumont~\cite[Section 2.1]{Dumont96} introduced the {\it Eulerian grammar}
$$G=\{x\rightarrow xy, y\rightarrow xy\},$$
and he proved that
\begin{equation*}
D^n(x)=x\sum_{k=0}^{n-1}\Eulerian{n}{k}x^{k}y^{n-k},
\end{equation*}
where $\Eulerian{n}{k}$ is the {\it Eulerian number}. The number $\Eulerian{n}{k}$ is closely related to $\Stirling{n}{k}$ (see~\cite[Eq.~(6.39)]{Graham89}):
\begin{equation}\label{StirlingEulerian}
k!\Stirling{n}{k}=\sum_{j}\Eulerian{n}{j}\binom{j}{n-k}.
\end{equation}
Recently, various extensions of the Eulerian grammar have been studied by several authors (see~\cite{Chen121,Chen122,Ma1301,Ma1302,Ma1303}). For example, Chen and Fu~\cite{Chen122} showed that the grammar
$$G=\{x\rightarrow x^2y, y\rightarrow x^2y\}$$
can be used to generate the {\it second-order Eulerian numbers} (see~\cite[A008517]{Sloane}).

Motivated by the similarity of~\eqref{Scherk-Stirling} and~\eqref{Chen-grams}, it is natural to
investigate the connection between context-free grammars and normal ordering problem, and then to explore various extensions of the Stirling grammar~\eqref{Chen-grammar}.
In this paper we present grammatical characterizations of several well known combinatorial sequences, including
the generalized Stirling numbers of the second kind related to the normal ordering problem and
the $r$-Dowling polynomials. Also, possible avenues for future research are described.

\section{Normal ordering problem and context-free grammars}
We can obtain the normal ordering $\mathcal{N}(\omega)$ of a word $\omega$ by means of {\em contractions} and {\em double dot operations}. The \emph{double dot operation} deletes all the letters $\varnothing$ and $\varnothing^{\dagger}$ in the word and then arranges it such that all the letters $b^{\dagger}$ precede the letters $b$. For example, $:b^{k}(b^{\dagger})^{l}:$ $=(b^{\dagger})^{l}b^{k}$. A \emph{contraction} consists of substituting $b=\varnothing$ and $b^{\dagger}=\varnothing^{\dagger}$ in the word whenever $b$ precedes $b^{\dagger}$. Among all possible contractions, we also include the null contraction, that is, the contraction leaving the word as
it is. The contents of {\it Wick's theorem} is the statement that
\begin{equation}\label{s0eq1}
\mathcal{N}(\omega)=\sum:\{\text{all possible contractions of $\omega$}\}:.
\end{equation}
An example for this is given by \eqref{boson-Stirling}, where the number of contractions of $(b^{\dagger}b)^n$ having exactly $n-k$ pairs of $b$ and $b^{\dagger}$ contracted is given by $\Stirling{n}{k}$.

Contractions can be depicted with diagrams called {\it linear representations}. Let us consider a word $\pi$ on the alphabet $\{b,b^{\dag}\}$ of length $n$, that is, $\pi=\pi_1\pi_{2}\cdots\pi_{n-1}\pi_n$ with $\pi_{i}\in\{b,b^{\dag}\}$. We draw $n$ vertices, say $1,2,\ldots,n$, on a horizontal line, such that the point $i$ corresponds to the letter $\pi_{i}$; we represent each $b$ by a white vertex and each letter $b^{\dagger}$ by a black vertex. A black vertex $j$ can be connected by an undirected edge $(i,j)$ to a white vertex $i$ when $i<j$ (but there may also be black vertices having no edge), where the edges are drawn in the plane above the points. This is the \emph{linear representation} of a contraction. An example is given in Figure~\ref{falt} for the word $(b^{\dagger}b)^2=b^{\dagger}bb^{\dagger}b$. A vertex having no edge is also called to be of {\it degree zero}.
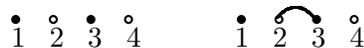
\begin{figure}[h]
\begin{center}
\begin{pspicture}(0,0)(5,.4)
\setlength{\unitlength}{3mm} \linethickness{0.3pt}
\multips(0,0)(3,0){2}{\pscircle*(0,0){.2}\pscircle(.5,0){.2}\pscircle*(1,0){.2}\pscircle(1.5,0){.2}}
\linethickness{0.8pt}
\qbezier(11.6,.1)(12.4,1)(13.2,.1)
%\qbezier(20,.1)(22.5,1)(25,.1)\
%qbezier(31.6,.1)(32.4,1)(33.2,.1)
\put(-.2,-1.2){$1$}\put(1.5,-1.2){$2$}\put(3.2,-1.2){$3$}\put(4.9,-1.2){$4$}
\put(9.7,-1.2){$1$}\put(11.4,-1.2){$2$}\put(13.1,-1.2){$3$}\put(14.8,-1.2){$4$}
%\put(19.7,-1.2){$4$}\put(21.4,-1.2){$3$}\put(23.1,-1.2){$2$}\put(24.8,-1.2){$1$}
%\put(29.7,-1.2){$4$}\put(31.4,-1.2){$3$}\put(33.1,-1.2){$2$}\put(34.8,-1.2){$1$}
\end{pspicture}
\end{center}
\caption{The linear representation of the contractions of the word $(b^{\dagger}b)^2$.}\label{falt}
\end{figure}

Let $G$ be the Stirling grammar~\eqref{Chen-grammar}. Let $\omega,\omega'$ be two words on the alphabet $\{x,y\}$. We denote the number of letters in $\omega$ by $|\omega|$, and we shall write $\omega=\omega^{(1)}\omega^{(2)}\cdots\omega^{(|\omega|)}$. We say that $(G^d,\omega)$ {\em generates $\omega'$ via the sequence $1=s_1,s_2,\ldots,s_{d+1}$} if there exists a sequence of words $a_n$ such that $a_1=\omega$, $a_{d+1}=\omega'$, and for all $j=2,3,\ldots,d+1$,
$$a_j=a_{j-1}^{(1)}\cdots a_{j-1}^{(s_j-1)}D(a_{j-1}^{(s_j)})a_{j-1}^{(s_j+1)}\cdots a_{j-1}^{(|a_{j-1}|)}.$$
Given the word $\omega$, the sequence $1=s_1,s_2,\ldots,s_{d+1}$ uniquely determines the word $\omega'$. However, given two words $\omega$ and $\omega'$, there may exist no (or more than one) such sequence $1=s_1,s_2,\ldots,s_{d+1}$ to obtain $\omega'$ from $\omega$.

\begin{example}\label{ExStir}
Let $G$ be the Stirling grammar~\eqref{Chen-grammar}.
For instance, $(G^2,xy)$ generates $xy^3$ via the sequence $1,1,1$; namely $a_1=xy$, $a_2=D(x)y=xyy$ and $a_3=D(x)yy=xyyy=xy^3$. Also, $(G^2,xy)$ generates $xy^2$ via the sequence $1,2,1$; namely $a_1=xy$, $a_2=xD(y)=xy$ and $a_3=D(x)y=xyy=xy^2$ (also via the sequences $1,1,2$ or $1,1,3$). As another example, $(G^2,xy)$ generates $xy$ via the sequence $1,2,2$; namely $a_1=xy$, $a_2=xD(y)=xy$ and $a_3=xD(y)=xy$.
\end{example}

\begin{theorem}
Let $G$ be the Stirling grammar~\eqref{Chen-grammar}.
There exists a bijection between the set of contractions of $(b^{\dagger}b)^{n+1}$ and the multiset of words that are generated by $(G^n,xy)$.
\end{theorem}
\begin{proof}
Let $(G^{n},xy)$ generate $xy^k$ via the sequence $s=s_1,s_2,\ldots,s_{n+1}$. Let $p_j$ the number of ones in the subsequence $s_1,s_2,\ldots,s_{j-1}$, for $j=2,\ldots,n+1$. Clearly, $s_j\leq p_j+1$ for all $j=1,2,\ldots,n+1$.
Associated to the sequence $s$ we define a contraction $C(s)$ on $2n+2$ vertices $1,2,\ldots,2n+2$, where the vertices $2j-1$ (resp. $2j$) are colored black (resp. white), for all $j=1,2,\ldots,n+1$, as follows. From left to right, for each $j=2,3,\ldots,n+1$, we do the following: If $s_j\neq 1$, then we connect the vertex $2j-1$ with the $(s_j-1)$-st white vertex, counted to the left side of $2j-1$, that has not been used (there exists such a vertex since $s_j\leq p_j+1$); if $s_j=1$, then the vertex $2j-1$ remains unconnected. Hence, $C(s)$ is a contraction of $(b^{\dagger}b)^{n+1}$.

On the other hand, if $C$ is a contraction of  $(b^{\dagger}b)^{n+1}$, then the sequence $s_1=1,s_2,\ldots,s_{n+1}$ can be defined as follows. We label each edge that connects a black vertex $v$ with a white vertex $w$ by the number of white vertices between $v$ and $w$ that are not connected to any black vertex $y<v$. Now, for each $j=2,3,\ldots,n+1$, we define $s_j=1$ if there exists no edge for the vertex $2j-1$, otherwise define $s_j$ by the requirement that $s_j-2$ equals the label of this edge. This implies that $s_1=1$ and $s_j\leq p_j+1$ for all $j=2,3,\ldots,n+1$, where $p_j$ counts the number of black vertices in the set $1,3,\ldots,2j-1$ of degree zero. Hence, the sequence $s_1,s_2,\ldots,s_{n+1}$ defines a word that is generated by $(G^n,xy)$, completing the proof.
\end{proof}

For instance, if $n=2$, then the multiset of words that is generated by $(G^2,xy)$ is given by $xy^3$; $xy^2$; $xy^2$; $xy^2$ and $xy$ via the sequences $1,1,1$; $1,1,2$; $1,1,3$; $1,2,1$ and $1,2,2$, respectively, see Example~\ref{ExStir}. The corresponding contractions are shown in Figure~\ref{f}.
\begin{figure}[h]
\begin{center}
\begin{pspicture}(0,0)(14,.4)
\setlength{\unitlength}{3mm} \linethickness{0.3pt}
\multips(0,0)(3,0){5}{\pscircle*(0,0){.2}\pscircle(.4,0){.2}\pscircle*(.8,0){.2}\pscircle(1.2,0){.2}\pscircle*(1.6,0){.2}\pscircle(2,0){.2}}
\linethickness{0.8pt}
\put(.1,0){\qbezier(13.9,.1)(14.6,1)(15.3,.1)}
\put(-2.6,0){\qbezier(23.9,.1)(26,1)(28.0,.1)}
\put(-2.2,0){\qbezier(33.5,.1)(34.2,1)(34.8,.1)}
\qbezier(43.9,.1)(44.6,1)(45.4,.1)\put(-2.6,0){\qbezier(43.9,.1)(44.6,1)(45.4,.1)}
\put(-.2,-1.2){$1$}\put(1.1,-1.2){$2$}\put(2.4,-1.2){$3$}\put(3.7,-1.2){$4$}\put(5.0,-1.2){$5$}\put(6.3,-1.2){$6$}
\put(9.7,-1.2){$1$}\put(11.1,-1.2){$2$}\put(12.4,-1.2){$3$}\put(13.7,-1.2){$4$}\put(15.0,-1.2){$5$}\put(16.3,-1.2){$6$}
\put(19.7,-1.2){$1$}\put(21.1,-1.2){$2$}\put(22.4,-1.2){$3$}\put(23.7,-1.2){$4$}\put(25.0,-1.2){$5$}\put(26.3,-1.2){$6$}
\put(29.7,-1.2){$1$}\put(31.1,-1.2){$2$}\put(32.4,-1.2){$3$}\put(33.7,-1.2){$4$}\put(35.0,-1.2){$5$}\put(36.3,-1.2){$6$}
\put(39.7,-1.2){$1$}\put(41.1,-1.2){$2$}\put(42.4,-1.2){$3$}\put(43.7,-1.2){$4$}\put(45.0,-1.2){$5$}\put(46.3,-1.2){$6$}
\end{pspicture}
\end{center}
\caption{The linear representation of the contractions of the word $(b^{\dagger}b)^3$.}\label{f}
\end{figure}
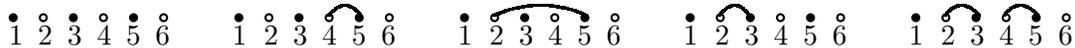
Note that there are $\Stirling{3}{3-k}$ contractions where exactly $k$ pairs are contracted. For instance, there is $\Stirling{3}{3}=1 $ contraction where no pairs are contracted, corresponding to the word $xy^3$. There are $\Stirling{3}{2}=3$ contractions where exactly one pair is contracted, corresponding to the word $xy^2$. Finally, there is $\Stirling{3}{1}=1$ contraction where exactly two pairs are contracted, corresponding to the word $xy$.

As a corollary to the proof of the above theorem, we obtain the following result.

\begin{corollary}
Let $P_n$ be the set of all sequences $s_1,s_2,\ldots,s_n$ such that $s_1=1$ and $s_j\leq |\{i|s_i=1,i<j\}|+1$, for all $j=2,3,\ldots,n$. Then the cardinality of $P_n$ is given by $B_n=\sum_{k=1}^n\Stirling{n}{k}$, the $n$-th Bell number. Moreover, the number of sequences in $P_n$ that contain exactly $k$ ones is given by $\Stirling{n}{k}$, the Stirling number of the second kind.
\end{corollary}

Before closing this section, let us remark that due to $D(x)=xy$ one has $(G^{n},xy)=(G^{n+1},x)$, showing that one has a bijection between the set of contractions of $(b^{\dagger}b)^{n}$ and the multiset of words that are generated by $(G^n,x)$.

\section{On the context-free grammar $G=\{x\rightarrow px+xy, y\rightarrow y\}$}\label{sec-2:generalized Stirling numbers}
Let $F(a,a^{\dag})$ be a possibly infinite word on the alphabet $\{a,a^{\dag}\}$.
Following~\cite{Mansour0702}, we define $\mathcal{C}(F(a,a^{\dag}))$ to be the multiset of all words obtained by substituting $a=e$ and $a^{\dag}=e^{\dag}$ whenever $a$ precedes $a^{\dag}$. Moreover, we replace any two adjacent letters $e$ and $e^{\dag}$ with $p$. For
each word $\pi\in \mathcal{C}(F(a,a^{\dag}))$, set $\widehat{\pi}=(a^{\dag})^ua^vp^\omega$ for some $u,v,\omega\geq 0$.

We define
$$\mn_p[F(a,a^{\dag})]=\sum_{\pi\in \mathcal{C}(F(a,a^{\dag}))}\widehat{\pi}.$$
According to~\cite{Mansour0702}, we have the following result.
\begin{theorem}
For all $n\geq 1$, we have
$$\mn_p[(a^{\dag}a)^n]=\sum_{k=0}^nS_p(n,k)(a^{\dag})^ka^k,$$
where $S_p(n,k)$ satisfies the recurrence relation
\begin{equation}\label{Spnk-recu}
S_p(n,k)=(k-1+p)S_p(n-1,k)+S_p(n-1,k-1)
\end{equation}
with the initial conditions $S_p(n,1)=p^{n-1}$ and $S_p(n,k)=0$ for $k>n$.
\end{theorem}

In the following theorem, a grammatical characterization of the numbers $S_p(n,k)$ is given.
\begin{theorem}\label{pStirling}
If  $G=\{x\rightarrow px+xy, y\rightarrow y\}$, then
$$D^{n-1}(x)=x\sum_{k=1}^{n}S_p(n,k)y^{k-1},\quad {\text for}\quad n\geq 2.$$
\end{theorem}
\begin{proof}
Note that $D(x)=x(p+y)$. We define $h(n,k)$ by
$$D^{n-1}(x)=x\sum_{k=1}^{n}h(n,k)y^{k-1}.$$
Since
$$D(D^{n-1}(x))=x\sum_k(k-1+p)h(n,k)y^{k-1}+x\sum_kh(n,k)y^k$$
it follows that
\begin{equation}\label{hnk-recu}
h(n+1,k)=(k-1+p)h(n,k)+h(n,k-1).
\end{equation}
By comparing~\eqref{hnk-recu} with~\eqref{Spnk-recu},  we see that the coefficients $h(n,k)$ satisfy the same recurrence relation and
initial conditions as $S_p(n,k)$, so they agree.
\end{proof}

Our next aim is to prove combinatorially Theorem~\ref{BijectionP}. In order to do that, we introduce the following notation and definitions. Let $G=\{x\rightarrow px+xy, y\rightarrow y\}$ be the context-free grammar. By induction, it is not hard to see that the monomials of $D^n(x)$ have the form $xy^k$ (we omit the coefficients), for $0\leq k\leq n$. For instance, $D^2(x)=p^2x+(2p+1)xy+xy^2$, which implies that the monomials of $D^2(x)$ are $x$, $xy$ and $xy^2$. We say that $(G^d,x)$ {\em generates $xy^k$ via the sequence $1=s_1,s_2,\ldots,s_{d+1}$} if there exists a sequence of words $a_n$ such that $a_1=x$, $a_{d+1}=xy^k$, and for all $j=2,3,\ldots,d+1$,
\begin{align*}
a_j=\left\{\begin{array}{ll}
xa_{j-1}^{(2)}\cdots a_{j-1}^{(|a_{j-1}|)},& s_j=1,\\
xya_{j-1}^{(2)}\cdots a_{j-1}^{(|a_{j-1}|)},& s_j=2,\\
a_{j-1}^{(1)}\cdots a_{j-1}^{(s_j-2)}a_{j-1}^{(s_j-1)}ya_{j-1}^{(s_j+1)}\cdots a_{j-1}^{(|a_{j-1}|)},& s_j\geq3.\\
\end{array}\right.
\end{align*}
Thus, $s_j=1$ means we choose the first letter - which is $x$ - and use the first part of the rule $x \rightarrow px+xy$ (namely $px$) in step $j$, $s_j=2$ means the same except that we choose the second part of the rule (namely $xy$). If $s_j=k$ with $k\geq 3$, then we choose the $(k-1)$-st letter  - which is $y$ - and use the second rule $y \rightarrow y$.

For instance, $(G^0,x)$ generates $x$ by the sequence $1$, $(G^1,x)$ generates $x$ by the sequence $1,1$ and generates $xy$ by $1,2$, and $(G^2,x)$ generates $x$ by the sequence $1,1,1$, and $xy$ by the sequences $1,1,2$, $1,2,1$ and $1,2,3$, and generates $xy^2$ by the sequence $1,2,2$. Clearly, each monomial $xy^k$ can be determined uniquely by its sequence, and the set of these sequences $s_1,s_2,\ldots,s_{n+1}$ satisfies $s_1=1$ and $1\leq s_j\leq q_j+2=|\{i|s_i=2,i\leq j-1\}|+2$, for all $j=2,3,\ldots,n+1$.

\begin{theorem}\label{BijectionP}
Let $G=\{x\rightarrow px+xy, y\rightarrow y\}$ be the context-free grammar. Then there exists a bijection between the set of contractions of $(b^{\dagger}b)^{n+1}$ with exactly $m$ edges that connect two adjacent vertices and with exactly $\ell$ black vertices of degree zero and the multiset of words that are generated by $(G^n,x)$ of the form $p^mxy^\ell$.
\end{theorem}
\begin{proof}
Let $(G^{n},x)$ generate $p^mxy^{\ell}$ via the sequence $s=s_1,s_2,\ldots,s_{n+1}$. Associated to the sequence $s$ we define a contraction $C(s)$ on $2n+2$ vertices $1,2,\ldots,2n+2$, where the vertices $2j-1$ (resp. $2j$) are colored black (resp. white), for all $j=1,2,\ldots,n+1$, as follows. From left to right, for each $j=2,3,\ldots,n+1$, we do the following: If $s_j\neq2$, then we connect the vertex $2j-1$ either with the first left white vertex of degree zero when $s_j=1$, or with the $(s_j-2)$-th white vertex, counted to the left side of $2j-1$, that has not been used (there exists such a vertex since $1\leq s_j\leq q_j+2$); if $s_j=2$, then the vertex $2j-1$ remains unconnected. Hence, $C(s)$ is a contraction of $(b^{\dagger}b)^{n+1}$. Note that after starting the generation from $x$, each time one adds a factor of $y$ means one uses the second part of $x \rightarrow px+xy$. Thus, $y^{\ell}$ means that there are $\ell$ 2's in the sequence $s$. Consequently, there are $\ell$ black vertices of degree zero. Similarly, each time one uses the first part of $x \rightarrow px+xy$, one adds a one to the sequence. Thus, $p^{m}$ means that there are $m$ ones in the sequence $s$. Consequently, there are $m$ edges connecting adjacent vertices in the contraction.

On the other hand, if $C$ is a contraction of  $(b^{\dagger}b)^{n+1}$ with exactly $m$ edges that connect two adjacent vertices and with exactly $\ell$ black vertices of degree zero, then the sequence $s_1=1,s_2,\ldots,s_{n+1}$ can be defined as follows. We label each edge that connects a black vertex $v$ with a white vertex $w$ by the number of white vertices between $v$ and $w$ that are not connected to any black vertex $y<v$. Now, for each $j=2,3,\ldots,n+1$, we define $s_j=2$ if there exists no edge for the vertex $2j-1$, $s_j=1$ if there exists an edge for the vertex $2j-1$ with label $0$, and otherwise define $s_j$ by the requirement that $s_j-2$ equals the label of this edge. This implies that $s_1=1$ and $1\leq s_j\leq q_j+2$ for all $j=2,3,\ldots,n+1$, where $q_j$ counts the number of black vertices in the set $1,3,\ldots,2j-1$ of degree zero. Hence, the sequence $s_1,s_2,\ldots,s_{n+1}$ defines a word that is generated by $(G^n,x)$. Now, it remains to characterize the coefficient of the monomial that is generated via the sequence $s=s_1,s_2,\ldots,s_{n+1}$. By the construction, the monomial is given by $xy^{|\{i|s_i=2,i=2,3,\ldots,n+1\}|}=xy^{\ell}$ with coefficient $p^{|\{i|s_i=1,i=2,3,\ldots,n+1\}|}=p^m$.
\end{proof}

In Figure~\ref{f22}, the linear representation of all $B(4)=15$ contractions of $(b^{\dag}b)^4$ are shown, together with the resulting sequences as defined in the proof of Theorem~\ref{BijectionP}. As an example, consider the contraction shown in the middle of the third row. Its associated sequence is $s=1,2,1,3$. Thus, there exist $\ell=1$ 2's in the sequence and $m=1$ ones, so the associated monomial is given by $pxy$. On the other hand, we can generate this monomial explicitly using the sequence 1,2,1,3. Associated to the sequence 1 is the monomial $x$, to 1,2 the monomial $xy$, to 1,2,1 the monomial $pxy$, and to 1,2,1,3 the monomial $pxy$, as expected.

\begin{figure}[h]
\begin{center}
\begin{pspicture}(0,-1)(11,.4)
\setlength{\unitlength}{3mm} \linethickness{0.3pt}
\multips(0,0)(4,0){3}{\pscircle*(0,0){.2}\pscircle(.4,0){.2}\pscircle*(.8,0){.2}\pscircle(1.2,0){.2}\pscircle*(1.6,0){.2}\pscircle(2,0){.2}\pscircle*(2.4,0){.2}\pscircle(2.8,0){.2}}
\linethickness{0.8pt}
\qbezier(1.3,.1)(2,1)(2.7,.1)\put(2.65,0){\qbezier(1.3,.1)(2,1)(2.7,.1)}\put(5.3,0){\qbezier(1.3,.1)(2,1)(2.7,.1)}
\put(13.3,0){\qbezier(1.3,.1)(2,1)(2.7,.1)\put(2.65,0){\qbezier(1.3,.1)(2,1)(2.7,.1)}}
\put(26.6,0){\qbezier(1.3,.1)(2,1)(2.7,.1)\put(5.38,0){\qbezier(1.3,.1)(2,1)(2.7,.1)}}
\put(-.2,-1.2){$1$}\put(1.1,-1.2){$2$}\put(2.4,-1.2){$3$}\put(3.7,-1.2){$4$}\put(5.0,-1.2){$5$}\put(6.3,-1.2){$6$}\put(7.6,-1.2){$7$}\put(8.9,-1.2){$8$}
\put(13.2,0){\put(-.2,-1.2){$1$}\put(1.1,-1.2){$2$}\put(2.4,-1.2){$3$}\put(3.7,-1.2){$4$}\put(5.0,-1.2){$5$}\put(6.4,-1.2){$6$}\put(7.7,-1.2){$7$}\put(9,-1.2){$8$}}
\put(26.6,0){\put(-.2,-1.2){$1$}\put(1.1,-1.2){$2$}\put(2.4,-1.2){$3$}\put(3.7,-1.2){$4$}\put(5.0,-1.2){$5$}\put(6.4,-1.2){$6$}\put(7.7,-1.2){$7$}\put(9,-1.2){$8$}}
\put(1.2,-2.4){$s=1,1,1,1$}\put(14.4,-2.4){$s=1,1,1,2$}\put(27.8,-2.4){$s=1,1,2,1$}
\end{pspicture}
\begin{pspicture}(0,-1)(11,.4)
\setlength{\unitlength}{3mm} \linethickness{0.3pt}
\multips(0,0)(4,0){3}{\pscircle*(0,0){.2}\pscircle(.4,0){.2}\pscircle*(.8,0){.2}\pscircle(1.2,0){.2}\pscircle*(1.6,0){.2}\pscircle(2,0){.2}\pscircle*(2.4,0){.2}\pscircle(2.8,0){.2}}
\linethickness{0.8pt}
\qbezier(1.3,.1)(2,1)(2.7,.1)
\put(13.3,0){\qbezier(1.3,.1)(2,1)(2.7,.1)\put(2.65,0){\qbezier(1.3,.1)(3.4,1)(5.5,.1)}}
\put(26.6,0){\put(2.65,0){\qbezier(1.3,.1)(2,1)(2.7,.1)}\put(5.3,0){\qbezier(1.3,.1)(2,1)(2.7,.1)}}
\put(-.2,-1.2){$1$}\put(1.1,-1.2){$2$}\put(2.4,-1.2){$3$}\put(3.7,-1.2){$4$}\put(5.0,-1.2){$5$}\put(6.3,-1.2){$6$}\put(7.6,-1.2){$7$}\put(8.9,-1.2){$8$}
\put(13.2,0){\put(-.2,-1.2){$1$}\put(1.1,-1.2){$2$}\put(2.4,-1.2){$3$}\put(3.7,-1.2){$4$}\put(5.0,-1.2){$5$}\put(6.4,-1.2){$6$}\put(7.7,-1.2){$7$}\put(9,-1.2){$8$}}
\put(26.6,0){\put(-.2,-1.2){$1$}\put(1.1,-1.2){$2$}\put(2.4,-1.2){$3$}\put(3.7,-1.2){$4$}\put(5.0,-1.2){$5$}\put(6.4,-1.2){$6$}\put(7.7,-1.2){$7$}\put(9,-1.2){$8$}}
\put(1.2,-2.4){$s=1,1,2,2$}\put(14.4,-2.4){$s=1,1,2,3$}\put(27.8,-2.4){$s=1,2,1,1$}
\end{pspicture}
\begin{pspicture}(0,-1)(11,.4)
\setlength{\unitlength}{3mm} \linethickness{0.3pt}
\multips(0,0)(4,0){3}{\pscircle*(0,0){.2}\pscircle(.4,0){.2}\pscircle*(.8,0){.2}\pscircle(1.2,0){.2}\pscircle*(1.6,0){.2}\pscircle(2,0){.2}\pscircle*(2.4,0){.2}\pscircle(2.8,0){.2}}
\linethickness{0.8pt}
\put(2.65,0){\qbezier(1.3,.1)(2,1)(2.7,.1)}
\put(13.2,0){\put(2.65,0){\qbezier(1.3,.1)(2,1)(2.7,.1)}\qbezier(1.3,.1)(4.8,1.6)(8.3,.1)}
\put(26.6,0){\put(5.3,0){\qbezier(1.3,.1)(2,1)(2.7,.1)}}
\put(-.2,-1.2){$1$}\put(1.1,-1.2){$2$}\put(2.4,-1.2){$3$}\put(3.7,-1.2){$4$}\put(5.0,-1.2){$5$}\put(6.3,-1.2){$6$}\put(7.6,-1.2){$7$}\put(8.9,-1.2){$8$}
\put(13.2,0){\put(-.2,-1.2){$1$}\put(1.1,-1.2){$2$}\put(2.4,-1.2){$3$}\put(3.7,-1.2){$4$}\put(5.0,-1.2){$5$}\put(6.4,-1.2){$6$}\put(7.7,-1.2){$7$}\put(9,-1.2){$8$}}
\put(26.6,0){\put(-.2,-1.2){$1$}\put(1.1,-1.2){$2$}\put(2.4,-1.2){$3$}\put(3.7,-1.2){$4$}\put(5.0,-1.2){$5$}\put(6.4,-1.2){$6$}\put(7.7,-1.2){$7$}\put(9,-1.2){$8$}}
\put(1.2,-2.4){$s=1,2,1,2$}\put(14.4,-2.4){$s=1,2,1,3$}\put(27.8,-2.4){$s=1,2,2,1$}
\end{pspicture}
\begin{pspicture}(0,-1)(11,.4)
\setlength{\unitlength}{3mm} \linethickness{0.3pt}
\multips(0,0)(4,0){3}{\pscircle*(0,0){.2}\pscircle(.4,0){.2}\pscircle*(.8,0){.2}\pscircle(1.2,0){.2}\pscircle*(1.6,0){.2}\pscircle(2,0){.2}\pscircle*(2.4,0){.2}\pscircle(2.8,0){.2}}
\linethickness{0.8pt}
\put(13.2,0){\put(2.65,0){\qbezier(1.3,.1)(3.4,1)(5.5,.1)}}
\put(26.6,0){\qbezier(1.3,.1)(4.8,1)(8.3,.1)}
\put(-.2,-1.2){$1$}\put(1.1,-1.2){$2$}\put(2.4,-1.2){$3$}\put(3.7,-1.2){$4$}\put(5.0,-1.2){$5$}\put(6.3,-1.2){$6$}\put(7.6,-1.2){$7$}\put(8.9,-1.2){$8$}
\put(13.2,0){\put(-.2,-1.2){$1$}\put(1.1,-1.2){$2$}\put(2.4,-1.2){$3$}\put(3.7,-1.2){$4$}\put(5.0,-1.2){$5$}\put(6.4,-1.2){$6$}\put(7.7,-1.2){$7$}\put(9,-1.2){$8$}}
\put(26.6,0){\put(-.2,-1.2){$1$}\put(1.1,-1.2){$2$}\put(2.4,-1.2){$3$}\put(3.7,-1.2){$4$}\put(5.0,-1.2){$5$}\put(6.4,-1.2){$6$}\put(7.7,-1.2){$7$}\put(9,-1.2){$8$}}
\put(1.2,-2.4){$s=1,2,2,2$}\put(14.4,-2.4){$s=1,2,2,3$}\put(27.8,-2.4){$s=1,2,2,4$}
\end{pspicture}
\begin{pspicture}(0,-1)(11,.4)
\setlength{\unitlength}{3mm} \linethickness{0.3pt}
\multips(0,0)(4,0){3}{\pscircle*(0,0){.2}\pscircle(.4,0){.2}\pscircle*(.8,0){.2}\pscircle(1.2,0){.2}\pscircle*(1.6,0){.2}\pscircle(2,0){.2}\pscircle*(2.4,0){.2}\pscircle(2.8,0){.2}}
\linethickness{0.8pt}
\qbezier(1.3,.1)(3.4,1)(5.5,.1)\put(5.3,0){\qbezier(1.3,.1)(2,1)(2.7,.1)}
\put(13.2,0){\qbezier(1.3,.1)(3.4,1)(5.5,.1)}
\put(26.6,0){\qbezier(1.3,.1)(3.4,1)(5.5,.1)\qbezier(4.1,.1)(6.2,1)(8.3,.1)}
\put(-.2,-1.2){$1$}\put(1.1,-1.2){$2$}\put(2.4,-1.2){$3$}\put(3.7,-1.2){$4$}\put(5.0,-1.2){$5$}\put(6.3,-1.2){$6$}\put(7.6,-1.2){$7$}\put(8.9,-1.2){$8$}
\put(13.2,0){\put(-.2,-1.2){$1$}\put(1.1,-1.2){$2$}\put(2.4,-1.2){$3$}\put(3.7,-1.2){$4$}\put(5.0,-1.2){$5$}\put(6.4,-1.2){$6$}\put(7.7,-1.2){$7$}\put(9,-1.2){$8$}}
\put(26.6,0){\put(-.2,-1.2){$1$}\put(1.1,-1.2){$2$}\put(2.4,-1.2){$3$}\put(3.7,-1.2){$4$}\put(5.0,-1.2){$5$}\put(6.4,-1.2){$6$}\put(7.7,-1.2){$7$}\put(9,-1.2){$8$}}
\put(1.2,-2.4){$s=1,2,3,1$}\put(14.4,-2.4){$s=1,2,3,2$}\put(27.8,-2.4){$s=1,2,3,3$}
\end{pspicture}
\end{center}
\caption{The linear representation of the contractions of the word $(b^{\dagger}b)^4$.}\label{f22}
\end{figure}
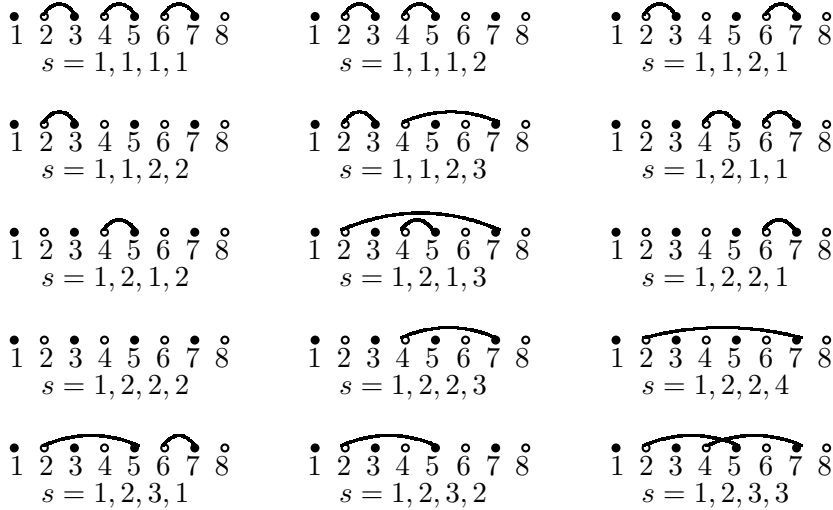

As a corollary to the proof of the above theorem, we obtain the following result.
\begin{corollary}
Let $Q_n$ be the set of all sequences $s_1,s_2,\ldots,s_{n}$ such that $s_1=1$ and $1\leq s_j\leq |\{i|s_i=2,i<j\}|+1$ for all $j=2,3,\ldots,n$. Then the cardinality of $Q_n$ is given by $\sum_{k=1}^{n}\Stirling{n}{k}$, the $n$-th Bell number. Moreover, the number of sequences in $Q_n$ that contain exactly $k$ 2's is given by $\Stirling{n}{k}$, the Stirling number of the second kind.
\end{corollary}

\section{On the context-free grammars $G_n=\{x\rightarrow (n-1)x+xy, y\rightarrow y\}$}\label{sec-4:generalized Stirling numbers}
The falling factorial is defined by $x^{\underline{n}}=\prod_{j=0}^{n-1}(x-j)$ with $x^{\underline{0}}=1$.
Blasiak et al.~\cite{Blasiak03} introduced the {\it generalized Stirling numbers of the second kind},  denoted $\Stirling{n}{k}_{r,s}$ for $r\geq s\geq 0$, by
\begin{align}\label{s1eq1}
\mn\{[(b^{\dag})^rb^s]^n\}=(b^{\dag})^{n(r-s)}\sum_{k=s}^{ns}\Stirling{n}{k}_{r,s}(b^{\dag})^kb^k.
\end{align}
It follows from~\eqref{boson-Stirling} that $\Stirling{n}{k}_{1,1}=\Stirling{n}{k}$. An equivalent form of the numbers $\Stirling{n}{k}_{r,s}$ has already been studied by Carlitz~\cite{Carlitz32} (and in the particular case $s=1$ even earlier by Scherk~\cite{S1823}). The reader is referred to~\cite{El-Desouky10}
for a thorough survey of the numbers $\Stirling{n}{k}_{r,s}$. Blasiak et al.~\cite{Blasiak03} obtained the following results:
\begin{equation}\label{Blasiak-identity}
e^{-x}\sum_{k=s}^{\infty}\frac{1}{k!}\prod_{j=1}^n(k+(j-1)(r-s))^{\underline{s}}x^k=\sum_{k=s}^{ns}\Stirling{n}{k}_{r,s}x^{k},
\end{equation}
\begin{equation*}
\prod_{j=1}^n(x+(j-1)(r-s))^{\underline{s}}=\sum_{k=s}^{ns}\Stirling{n}{k}_{r,s}x^{\underline{k}}.
\end{equation*}
In particular,
$$(x^{\underline{r}})^n=\sum_{k=r}^{nr}\Stirling{n}{k}_{r,r}x^{\underline{k}}.$$
They also found that the numbers $\Stirling{n}{k}_{r,r}$ satisfy fo $n>1$ the recurrence relation
\begin{equation}\label{Blasiak-recurrence relation}
\Stirling{n+1}{k}_{r,r}=\sum_{p=0}^r\binom{k+p-r}{p}r^{\underline{p}}\Stirling{n}{k+p-r}_{r,r}
\quad\textrm{for $r\leq k\leq nr$}
\end{equation}
with the initial conditions
$\Stirling{1}{r}_{r,r}=1$ and $\Stirling{n}{k}_{r,r}=0$ for $~k<r$ or $nr<k\leq (n+1)r$.

We define a sequence of grammars $\{G_n\}_{n\geq1}$ by
\begin{equation}\label{G-Stirling}
G_n=\{x\rightarrow (n-1)x+xy, y\rightarrow y\}.
\end{equation}
Let $D_n$ be the linear operator associated to the grammar $G_n$. The following theorem shows that the numbers $\Stirling{n}{k}_{r,1}$ and $\Stirling{n}{k}_{r,r}$ can be
generated by the grammars $G_n$.
%Blasiak, Penson and Salomon~\cite[Table~1]{Blasiak03} for several triangles of $\Stirling{n}{k}_{r,s}$.
\begin{theorem}\label{thm2}
When $r\geq 2$, we have
\begin{equation}\label{S-1}
D_{(n-1)r-(n-2)}D_{(n-2)r-(n-3)} \cdots D_{3r-2}D_{2r-1}D_rD_1(x)=x\sum_{k=1}^n\Stirling{n}{k}_{r,1}y^k,
\end{equation}
\begin{equation}\label{S-2}
 D_1(D_rD_{r-1}\cdots D_1)^{n-1}(x)=x\sum_{k=r}^{nr}\Stirling{n}{k}_{r,r}y^{k-r+1}.
\end{equation}
\end{theorem}
\begin{proof}
It follows from~\eqref{Blasiak-identity} that
$$\Stirling{n}{k}_{r,1}=[k+(n-1)(r-1)]\Stirling{n-1}{k}_{r,1}+\Stirling{n-1}{k-1}_{r,1},$$
with initial conditions $\Stirling{1}{k}_{r,1}=\delta_{k,1}$.
Note that $D_1(x)=xy$ and $D_rD_1(x)=x(ry+y^2)$. For $n\geq 1$, we define $g(n,k)$ by
$$D_{(n-1)r-(n-2)}D_{(n-2)r-(n-3)} \cdots D_{3r-2}D_{2r-1}D_rD_1(x)=x\sum_{k=1}^ng(n,k)y^k.$$
Using
$$D_{nr-(n-1)}\left\{x\sum_{k=1}^ng(n,k)y^k\right\}=[(nr-n)x+xy]\sum_{k=1}^ng(n,k)y^k+x\sum_{k=1}^nkg(n,k)y^k,$$
we obtain
$$g(n+1,k)=(k+nr-n)g(n,k)+g(n,k-1).$$
Hence, the numbers $g(n,k)$ satisfy the same recurrence relation
and initial conditions as $\Stirling{n}{k}_{r,1}$, so they agree. Using
$$ D_1(D_rD_{r-1}\cdots D_1)^{n}(x)=D_1D_rD_{r-1}\cdots D_2\left\{D_1(D_rD_{r-1}\cdots D_1)^{n-1}(x)\right\},$$
one can show along the same lines the corresponding assertion~\eqref{S-2}. This is a straightforward, albeit tedious, application of~\eqref{Blasiak-recurrence relation}.
\end{proof}

From \eqref{S-2} we read off that the number of terms in $D_1(D_rD_{r-1}\cdots D_1)^{n-1}(x)$ is given by $B_r(n)=\sum_{k=r}^{nr}\Stirling{n}{k}_{r,r}$, the $n$-th generalized Bell number.

%The reader is referred to~\cite[Table~1]{Blasiak03} for some triangles of generalized Stirling numbers $\Stirling{n}{k}_{r,s}$.
\begin{example}
When $r=3$, we have the following relations:
\begin{align*}
D_1(x)& =xy, \\
D_3D_1(x)& =x(3y+y^2), \\
D_5D_3D_1(x)& =x(15y+9y^2+y^3),\\
D_1(D_3D_2D_1)(x)&=x(6y+18y^2+9y^3+y^4).
\end{align*}
The last line shows that $B_3(2)=34$, see A069223 in~\cite{Sloane}.
\end{example}

Recall that a board $B$ is a subset of cells of an $n\times n$ chessboard. The {\it rook number} $r_k(B)$ is defined as the number of ways to put $k$ non-attacking rooks on the board $B$, and the corresponding generating function is called the {\it rook polynomial}. We define a {\it Ferrers board} to be a board with column heights given by $0\leq h_1\leq h_2\leq\cdots\leq h_n$, and denote it by $F(h_1,h_2,\ldots,h_n)$.
Let $G_n$ be as in~\eqref{G-Stirling}. For $n\geq 1$, we define $a_{n}(y)$ and $b_{n}(y)$ by
\begin{align*}
(D_2D_1)^n(x)&=xa_{n}(y),\\
D_1(D_2D_1)^{n-1}(x)&=xb_{n}(y).
\end{align*}
Note that the array of coefficients of the polynomials $a_{n}(y)$ is A088960 in~\cite{Sloane}. Moreover, the coefficients of the polynomials $b_{n}(y)$ are given by $\Stirling{n}{k}_{2,2}$ (see~\cite[A078739]{Sloane}). It is easy to verify that $a_{n}(y)$ is the rook polynomial of the Ferrers board
$F(1,1,3,3,\ldots,2n-3,2n-3,2n-1,2n-1)$, and $b_{n}(x)$ is the rook polynomial of the Ferrers board $F(1,1,3,3,\ldots,2n-3,2n-3,2n-1,2n-1,2n)$.
%By comparing the above Theorem \ref{thm2} and (\ref{s1eq1}) with $s=1$ or $s=r$, we obtain
%\begin{align*}
%[(b^{\dag})^rb]^n=\mn\{[(b^{\dag})^rb]^n\}=(b^{\dag})^{n(r-1)}\sum_{k=1}^{n}\Stirling{n}{k}_{r,1}(b^{\dag})^kb^k.
%\end{align*}
%\begin{equation}\label{S-1}
%D_{(n-1)r-(n-2)}D_{(n-2)r-(n-3)} \cdots D_{3r-2}D_{2r-1}D_rD_1(x)=x\sum_{k=1}^n\Stirling{n}{k}_{r,1}y^k,
%\end{equation}
%\begin{align*}
%[(b^{\dag})^rb^r]^n=\mn\{[(b^{\dag})^rb^r]^n\}=\sum_{k=r}^{nr}\Stirling{n}{k}_{r,r}(b^{\dag})^kb^k.
%\end{align*}
%\begin{equation}\label{S-2}
% D_1(D_rD_{r-1}\cdots D_1)^{n-1}(x)=x\sum_{k=r}^{nr}\Stirling{n}{k}_{r,r}y^{k-r+1}.
%\end{equation}

\section{On the context-free grammar $G_n=\{x\rightarrow q^n x+xy, y\rightarrow y\}$}\label{sec-5:generalized Stirling numbers}
The {\it $q$-analog of the Stirling numbers of the second kind} $\Stirling{n}{k}_q$ is defined by
$$(x+1)(x+q)\cdots (x+q^{n-1})=\sum_{k=1}^n\Stirling{n}{k}_q(x+1)^{\underline{k}}.$$

Using weighted partitions, Cigler~\cite{Cigler92} presented a combinatorial interpretation of the numbers $\Stirling{n}{k}_q$ and showed that they satisfy the recurrence relation
\begin{equation}\label{Snkq-recu}
\Stirling{n+1}{k}_q=(k-1+q^n)\Stirling{n}{k}_q+\Stirling{n}{k-1}_q
\end{equation}
with the initial conditions $\Stirling{1}{k}_q=\delta_{1,k}$. This recurrence relation gives rise to the following result.

\begin{theorem}
Let $G_n=\{x\rightarrow q^nx+xy, y\rightarrow y\}$, and let $D_n$ be the linear operator associated to the grammar $G_n$. Then we have, for $n\geq 2$, that
$$D_{n-1}\cdots D_3D_2D_1(x)=x\sum_{k=1}^{n}\Stirling{n}{k}_qy^{k-1}.$$
\end{theorem}
\begin{proof}
Note that $D_1(x)=x(q+y)$ and $D_2D_1(x)=x[q^3+(q^2+q+1)y+y^2]$.
We define $i(n,k)$ by
$$D_{n-1}\cdots D_3D_2D_1(x)=x\sum_{k=1}^{n}i(n,k)y^{k-1}.$$
Note that $$D_n(D_{n-1}\cdots D_3D_2D_1(x))=x\sum_k(k-1+q^n)i(n,k)y^{k-1}+x\sum_ki(n,k)y^k.$$
Therefore, we conclude that
$$i(n+1,k)=(k-1+q^n)i(n,k)+i(n,k-1),$$
and complete the proof by comparing this with~\eqref{Snkq-recu}.
\end{proof}

\section{On the context-free grammars $G=\{x\rightarrow rx+xy, y\rightarrow my\}$}\label{sec:Whitney numbers}
Let $mx+r$ be an arithmetric progression.
The {\it $r$-Whitney numbers
of the second kind} $W_{m,r}(n,k)$ are defined by
$$(mx+r)^n=\sum_{k=0}^nm^kW_{m,r}(n,k)x^{\underline{k}}$$
(see~\cite{Mezo10}). The exponential generating functions of the numbers $W_{m,r}(n,k)$ are given by
$$\sum_{n\geq k}W_{m,r}(n,k)\frac{z^n}{n!}=\frac{e^{rz}}{k!}\left(\frac{e^{mz}-1}{m}\right)^k.$$
The numbers $W_{m,r}(n,k)$ are a common generalization of the numbers $\Stirling{n}{k}_r$, the $r$-{\it Stirling numbers of the second kind} (see~\cite{Broder84}) and $W_m(n,k)$, the {\it Whitney numbers of the second kind} (see~\cite{Benoumhani97}).
More precisely, we have
\begin{equation*}
\begin{split}
W_{1,0}(n,k)&=\Stirling{n}{k},\\
W_{1,r}(n,k)&=\Stirling{n+r}{k+r}_r,\\
W_{m,1}(n,k)&=W_m(n,k).
\end{split}
\end{equation*}

The {\it $r$-Dowling polynomial} of degree $n$ is defined by
$$D_{m,r}(n;x)=\sum_{k=0}^nW_{m,r}(n,k)x^k.$$
It is well known that the numbers $W_{m,r}(n,k)$ satisfy the recurrence relation
$$W_{m,r}(n,k)=(r+km)W_{m,r}(n-1,k)+W_{m,r}(n-1,k-1)$$
with the initial conditions $W_{m,r}(1,0)=r$, $W_{m,r}(1,1)=1$ and  $W_{m,r}(1,k)=0$ for $k\geq2$. This is equivalent to
\begin{equation*}
D_{m,r}(n;x)=(r+x)D_{m,r}(n-1;x)+mxD_{m,r}'(n-1;x)
\end{equation*}
with $D_{m,r}(1;x)=r+x$ (see~\cite{Cheon12}). In the next theorem, a grammatical interpretation of the $r$-Dowling polynomials is given.
\begin{theorem}\label{Dowling}
If $G=\{x\rightarrow rx+xy, y\rightarrow my\}$, then
\begin{equation*}
D^n(x)=xD_{m,r}(n;y).
\end{equation*}
\end{theorem}
\begin{proof}
Note that $D(x)=x(r+y)$ and $D^2(x)=x[r^2+(m+2r)y+y^2]$. For $n\geq 1$, we define $D^n(x)=xf_n(y)$.
Then $f_1(y)=D_{m,r}(1;y)$. Moreover, since
$$D^{n+1}(x)=D(xf_n(y))=x(r+y)f_n(y)+mxyf_n'(y),$$
it follows that $f_{n+1}(y)=(r+y)f_n(y)+myf_n'(y)$.
Hence, $f_n(y)$ satisfy the same recurrence relation
and initial conditions as $D_{m,r}(n;y)$, so they agree.
\end{proof}

The grammar of Theorem~\ref{Dowling} coincides with the one of Theorem~\ref{pStirling} for $r=p$ and $m=1$, so there should be a close connection between the $S_p(n,k)$ and $W_{m,r}(n,k)$. In fact, writing in Theorem~\ref{Dowling} for $m=1$ and $r=p$ the result as $D^n(x)=xD_{1,p}(n;y)=x\sum_{k=0}^nW_{1,p}(n,k)y^k$ and comparing this to Theorem~\ref{pStirling}, one finds the relation
$$ W_{1,p}(n,k)=S_p(n+1,k+1). $$

It follows from~\eqref{Dnab-Leib} that $$D^{n+1}(x)=D^n(x+xy)=D^n(x)+\sum_{k=0}^n\binom{n}{k}D^k(x)D^{n-k}(y).$$
Clearly, we have $D^k(y)=m^ky$ for $k\geq 0$. Thus, we immediately obtain the following corollary.
\begin{corollary}[{\cite[Thm.~5.1]{Cheon12}}] The $r$-Dowling polynomials satisfy the recurrence relation
$$D_{m,r}(n+1;x)=rD_{m,r}(n;x)+x\sum_{k=0}^{n}\binom{n}{k}m^{n-k}D_{m,r}(k;x).$$
\end{corollary}

%%%%%%%%%%%%%%%%%%%%
\section{On the Stirling-Frobenius subset numbers}\label{Section-7}
%\hspace*{\parindent}
%%%%%%%%%%%%%%%%%%%%%%%%%%%%%%%%%%%%%%%%%%
Let $\sgn$ denote the sign function defined on $\mathbb{R}$, i.e.,
\begin{equation*}
\sgn x=\begin{cases}
+1 & \text{if $x>0$},\\ 0  & \text{if
$x=0$},\\ -1&\text{if
$x<0$.}
\end{cases}
\end{equation*}
Let
$$\Eulerian{n}{k}_m=\frac{1}{2}\sum_{j=0}^{n+1}(-1)^j\binom{n+1}{j}(m(k-j)+1)^n\sgn(m(k-j)+1),$$
with the special value $\Eulerian{0}{0}_1=1$. Luschny~\cite{Luschny} introduced the {\it Stirling-Frobenius subset numbers}, denoted $\Stirling{n}{k}_m$, by
\begin{equation*}\label{StirlingFrobenius}
\Stirling{n}{k}_m=\frac{1}{m^kk!}\sum_{j}\Eulerian{n}{j}_m\binom{j}{n-k},
\end{equation*}
compare this to~\eqref{StirlingEulerian}. Define $$\overline{\Stirling{n}{k}}_m=m^k\Stirling{n}{k}_m,$$
          $$\widetilde{\Stirling{n}{k}}_m=m^kk!\Stirling{n}{k}_m.$$
Luschny~\cite{Luschny} obtained the following recurrence relations:
\begin{equation*}
\begin{split}
\Stirling{n}{k}_m&=(m(k+1)-1)\Stirling{n-1}{k}_m+\Stirling{n-1}{k-1}_m,\\
\overline{\Stirling{n}{k}}_m&=(m(k+1)-1)\overline{\Stirling{n-1}{k}}_m+m\overline{\Stirling{n-1}{k-1}}_m,\\
\widetilde{\Stirling{n}{k}}_m&=(m(k+1)-1)\widetilde{\Stirling{n-1}{k}}_m+mk\widetilde{\Stirling{n-1}{k-1}}_m,\\
\end{split}
\end{equation*}
with the initial conditions $\Stirling{n}{0}_m=\overline{\Stirling{n}{0}}_m=\widetilde{\Stirling{n}{0}}_m=\delta_{n,0}$. In the following theorem, a grammatical characterization of the numbers $\Stirling{n}{k}_m,\overline{\Stirling{n}{k}}_m$ and $\widetilde{\Stirling{n}{k}}_m$  is given.
\begin{theorem}
Let $m\in\mathbb{N}$.
\begin{enumerate}
\item [(a)]
If $G=\{x\rightarrow (m-1)x+xy, y\rightarrow my\}$, then
\begin{equation}\label{FS11}
D^n(x)=x\sum_{k=0}^n\Stirling{n}{k}_my^k.
\end{equation}
\item [(b)]
If $G=\{x\rightarrow (m-1)x+mxy, y\rightarrow my\}$, then
\begin{equation}\label{FS12}
D^n(x)=x\sum_{k=0}^n\overline{\Stirling{n}{k}}_my^k.
\end{equation}
\item [(c)]
If $G=\{x\rightarrow (m-1)x+mxy, y\rightarrow m(y+y^2)\}$, then
\begin{equation}\label{FS13}
D^n(x)=x\sum_{k=0}^n\widetilde{\Stirling{n}{k}}_my^k.
\end{equation}
%$(x-1)^{n+q+1}\|A_n(x;q)$.
\end{enumerate}
\end{theorem}
\begin{proof}
We only prove the assertion~\eqref{FS11}, the assertions~\eqref{FS12} and ~\eqref{FS13} can be proved in a similar way.
Note that $D(x)=x(m-1+y)$ and $D^2(x)=x[(m-1)^2+(3m-2)y+y^2]$.
For $n\geq 1$, we define $j(n,k)$ by
$$D^n(x)=x\sum_{k=0}^nj(n,k)y^k.$$
Since
$$D(D^n(x))=x\sum_{k}(m(k+1)-1)j(n,k)y^k+x\sum_{k}j(n,k)y^{k+1}$$
it follows that
$$j(n+1,k)=(m(k+1)-1)j(n,k)+j(n,k-1).$$
Hence, the numbers $j(n,k)$ satisfy the same recurrence relation
and initial conditions as $\Stirling{n}{k}_m$, so they agree.
\end{proof}
%%%%%%%%%%%%%%%%%%%%
\section{Concluding remarks}\label{Section-5}
%\hspace*{\parindent}
%%%%%%%%%%%%%%%%%%%%%%%%%%%%%%%%%%%%%%%%%%
In this paper, we considered several extensions of the Stirling grammar~\eqref{Chen-grammar}. In fact, there are many other extensions of the Stirling grammar. Here we present two further examples.
\begin{example}
If $G=\{x\rightarrow xy+xy^2, y\rightarrow y^2\}$, then
\begin{equation*}
D^n(x)=xy^n\sum_{k=0}^nk!{\binom{n}{k}}^2y^{n-k}.
\end{equation*}
\end{example}
\begin{example}
If $G=\{x\rightarrow xy+xy^2, y\rightarrow y^3\}$, then
\begin{equation*}
D^n(x)=xy^n\theta_n(y),
\end{equation*}
where $\theta_n(x)=\sum_{k=0}^n\frac{(n+k)!}{(n-k)!k!}\left(\frac{x}{2}\right)^k$ is
the Bessel polynomial (see~\cite[A001498]{Sloane}).
%$(x-1)^{n+q+1}\|A_n(x;q)$.
\end{example}

Recall that for the conventional derivative $D=\frac{d}{dx}$ one may consider the {\it shift operator} $e^{\lambda D}$, which acts on an analytical function $f$ according to Taylor's theorem by $(e^{\lambda D}f)(x)=f(x+\lambda)$, in particular, $e^{\lambda D}(x)=x+\lambda$. Let us consider a context-free grammar $G$ and the associated formal derivative $D$. Define $e^{\lambda D}$ by the usual series, that is, $e^{\lambda D}=\sum_{n\geq 0} \frac{\lambda^n}{n!}D^n$. It would be interesting to find explicit expressions for $e^{\lambda D}(\omega)$ for some simple words $\omega$, for example $\omega=x$ or $\omega=xy$.
\begin{example}
Let $G=\{x\rightarrow xy, y\rightarrow y\}$ be the Stirling grammar \eqref{Chen-grammar} and $D$ the associated formal derivative. Since $D^n(y)=y$ for all $n=0,1,2,\ldots$, one finds $e^{\lambda D}(y)=\sum_{n\geq 0} \frac{\lambda^n}{n!}D^n(y)=e^{\lambda}y$. More generally, a simple induction shows that $D^n(y^m)=m^ny^m$, implying $e^{\lambda D}(y^m)=e^{\lambda m}y^m$. Using \eqref{Chen-grams}, one obtains
\begin{equation}\label{Shift}
e^{\lambda D}(x)=\sum_{n\geq 0} \frac{\lambda^n}{n!}D^n(x)=x \sum_{n\geq 0}\sum_{k\geq 0} \Stirling{n}{k} y^k\frac{\lambda^n}{n!}=xe^{(e^{\lambda}-1)y},
\end{equation}
which was already noted by Chen~\cite[Eq. (4.9)]{Chen93} from a different point of view. Slightly more involved is the evaluation of $e^{\lambda D}(xy)$. Recalling $D(x)=xy$, we find $e^{\lambda D}(xy)=e^{\lambda D}(D(x))=D(e^{\lambda D}(x))=D(xe^{(e^{\lambda}-1)y})$, where we have used \eqref{Shift}. Applying the formal product and chain rule, this gives $e^{\lambda D}(xy)=D(x)e^{(e^{\lambda}-1)y}+xD(e^{(e^{\lambda}-1)y})=xye^{(e^{\lambda}-1)y}+x(e^{\lambda}-1)e^{(e^{\lambda}-1)y}y$. Thus, we have shown that
\begin{equation}\label{Shift2}
e^{\lambda D}(xy)=xye^{\lambda}e^{(e^{\lambda}-1)y}.
\end{equation}
The determination of $e^{\lambda D}(\omega)$, where the word $\omega$ contains more than one $x$ (for example, $\omega=x^2$), seems to be more difficult.
\end{example}
\noindent However, it seems that for most grammars even the determination of $e^{\lambda D}(x)$ is a challenging task.

We end our paper by proposing the following open problem.
\begin{problem}
Let $G=\{x\rightarrow f(x,y),y\rightarrow g(x,y)\}$. Find a normal ordering problem that is equivalent to the context-free grammar $G$. For example, we showed that $G=\{x\rightarrow x+xy,y\rightarrow y\}$ is equivalent to the normal ordering problem where $xy-yx=1$.
\end{problem}
%Let $q$ be an indeterminate. Given two polynomials $f(q)$ and $g(q)$,
%we write $f(q)\ge_q g(q)$ if the difference has nonnegative coefficients as a polynomial of $q$.
%A sequence of polynomials $\{f_n(q)\}_{n\geq 0}$ over the field of real numbers is called
%{\it q-log-convex} if $$f_{n-1}(q)f_{n+1}(q)\ge_q f_n^2(q)\quad\textrm{for $n\ge 1$}.$$

%We end our paper by proposing the following.
%\begin{problem}
%It is well known that the Bell polynomials, the Eulerian polynomials, the Dowling polynomials and the Bessel %polynomials are q-log-convex~\cite{Chen11,Liu07}. Can either inequality be given a grammatical proof?
%\end{problem}
%%%%%%%%%%%%%%%%%%%%%%%%%%%%%%%%%%%%%%%%%%%
%\section{Results}\label{results}
%\hspace*{\parindent}
%%%%%%%%%%%%%%%%%%%%%%%%%%%%%%%%%%%%%%%%%%

\end{document}